\documentclass [twoside,reqno, 12pt] {amsart}

\usepackage{amsfonts}
\usepackage{amssymb}
\usepackage{a4}
\usepackage{color}

\newtheorem{thm}{Theorem}[section]

\newtheorem{lem}[thm]{Lemma}
\newtheorem{prop}[thm]{Proposition}

\theoremstyle{definition}

\numberwithin{equation}{section}

\renewcommand{\Re}{\hbox{Re}\,}
\renewcommand{\Im}{\hbox{Im}\,}

\newcommand{\C}{\mathbb{C}}

\newcommand{\R}{\mathbb{R}}

\newcommand{\supp}{\operatorname{supp}}

\def\hat{\widehat}
\def\tilde{\widetilde}
\def \bfo {\begin {eqnarray*} }
\def \efo {\end {eqnarray*} }
\def \ba {\begin {eqnarray*} }
\def \ea {\end {eqnarray*} }
\def \beq {\begin {eqnarray}}
\def \eeq {\end {eqnarray}}
\def \supp {\hbox{supp }}

\def \p {\partial}

\def\hat{\widehat}
\def\tilde{\widetilde}
\def \bfo {\begin {eqnarray*} }
\def \efo {\end {eqnarray*} }
\def \ba {\begin {eqnarray*} }
\def \ea {\end {eqnarray*} }
\def \beq {\begin {eqnarray}}
\def \eeq {\end {eqnarray}}
\def \supp {\hbox{supp }}

\def \p {\partial}

\begin{document}

 \title[Determining a magnetic Schr\"odinger operator]{Determining a magnetic Schr\"odinger operator with a continuous magnetic potential from boundary measurements}

\author[Krupchyk]{Katsiaryna Krupchyk}

\address
        {K. Krupchyk, Department of Mathematics and Statistics \\
         University of Helsinki\\
         P.O. Box 68 \\
         FI-00014   Helsinki\\
         Finland}

\email{katya.krupchyk@helsinki.fi}

\author[Uhlmann]{Gunther Uhlmann}

\address
       {G. Uhlmann, Department of Mathematics\\
       University of Washington\\
       Seattle, WA  98195-4350\\
       and
       Department of Mathematics\\ 
       340 Rowland Hall\\
        University of California\\
        Irvine, CA 92697-3875\\
       USA}
\email{gunther@math.washington.edu}

\maketitle

\begin{abstract} 

We show that the knowledge of the set of the Cauchy data on the boundary of a $C^1$ bounded open set in $\R^n$, $n\ge 3$,  for the Schr\"odinger operator with continuous magnetic and bounded electric potentials determines the magnetic field and electric potential inside the set uniquely.  The proof is based on a Carleman estimate for the magnetic Schr\"odinger operator with a gain of two derivatives.

\end{abstract}

\section{Introduction and statement of result}

Let $\Omega\subset \R^n$, $n\ge 3$, be a bounded open set with $C^1$ boundary, and let $u\in C^\infty_0(\Omega)$.  We consider the magnetic Schr\"odinger operator,
\begin{align*}
L_{A,q}(x,D)&u(x):=\sum_{j=1}^n (D_j+A_j(x))^2u(x)+q(x)u(x)\\
&=-\Delta u(x) +A(x)\cdot Du(x) +D\cdot (A(x)u(x))   +((A(x))^2+q(x))u(x),
\end{align*}
where $D=i^{-1}\nabla$, $A\in L^\infty(\Omega,\C^n)$ is the magnetic potential, and $q\in L^\infty(\Omega,\C)$ is the electric potential. Notice that here $Au\in L^\infty(\Omega,\C^n)\cap\mathcal{E}'(\Omega,\C^n)$, and therefore, 
\[
L_{A,q}:C^\infty_0(\Omega)\to H^{-1}(\R^n)\cap\mathcal{E}'(\Omega)
\]
is a bounded operator. 

Let $u\in H^1(\Omega)$ be a solution to 
\begin{equation}
\label{eq_1_1}
L_{A,q}u=0\quad \textrm{in}\quad \Omega,
\end{equation}
in the sense of distributions. Denoting by $\nu$ the unit outer normal to the boundary of $\Omega$, we can define the trace of $\p_\nu u+i(A\cdot\nu)u$ on $\p\Omega$ as an element of $H^{-1/2}(\p \Omega)$ as follows.  Let $g\in H^{1/2}(\p \Omega)$. Then there is a continuous extension of $g$ to $\Omega$, which belongs to $H^1(\Omega)$, and which we shall denote again by $g$.  We set 
\begin{equation}
\label{eq_trace}
\begin{aligned}
\langle \p_\nu u + i(A\cdot\nu) u, g\rangle_{\p \Omega}:=\int_{\Omega} (\nabla u\cdot\nabla g + iA\cdot (u\nabla g-g\nabla u)+ (A^2+q)u g)\,dx.
\end{aligned}
\end{equation}
As $u$ satisfies the equation \eqref{eq_1_1}, the above definition of the trace of $\p_\nu u+i(A\cdot\nu)u$ on $\p\Omega$ is independent of the choice of an extension of $g$.

We introduce the set of the Cauchy data for solutions of the magnetic Schr\"odinger equation in $\Omega$ by 
\[
C_{A,q}:=\{(u|_{\p \Omega}, (\p_\nu u + i(A\cdot\nu) u)|_{\p \Omega}): u\in H^1(\Omega)\textrm{ and }  L_{A,q}u=0 \textrm{ in } \Omega\}.
\]
The inverse boundary value problem for the magnetic Schr\"odinger operator is to determine $A$ and $q$ in $\Omega$ from the set of the Cauchy data $C_{A,q}$.

As it was noticed in \cite{Sun_1993}, there is an obstruction to uniqueness given by the following gauge transformation. Let $\psi\in W^{1,\infty}(\Omega)$. Then for any $u\in H^1(\Omega)$, we have 
\[
e^{-i\psi}L_{A,q}e^{i\psi}u=L_{A+\nabla \psi,q}u,
\]
in the sense of distributions. Furthermore, a computation using \eqref{eq_trace}
shows that 
\[
\langle e^{-i\psi}(\p_\nu+i(A\cdot \nu))e^{i\psi}u,g\rangle_{\p \Omega} = \langle (\p_\nu+i(A+\nabla\psi)\cdot \nu)u,g\rangle_{\p \Omega}
\]
for any $g\in H^{1/2}(\p \Omega)$.  Hence, if $\psi\in W^{1,\infty}(\Omega)$ satisfies $\psi|_{\p \Omega}=0$, then 
\begin{equation}
\label{eq_invariance_cauchy}
C_{A,q}=C_{A+\nabla\psi,q}.
\end{equation} 
Thus, given the set of the Cauchy data $C_{A,q}$ for the magnetic Schr\"odinger operator,  one may only hope to recover the magnetic field $dA$ in $\Omega$, which is defined by 
\[
dA=\sum_{1\le j<k\le n}(\p_{x_j}A_k-\p_{x_k}A_j)dx_j\wedge dx_k
\]
in the sense of distributions.  Here $A=(A_1,\dots,A_n)$.  

As it has been shown by several authors, the knowledge of the set of the Cauchy data $C_{A,q}$ for the magnetic Schr\"odinger operator $L_{A,q}$ determines the magnetic field and the electric potential in $\Omega$ uniquely, 
under certain regularity assumptions on $A$. In \cite{Sun_1993}, this result was established for magnetic potentials in $W^{2,\infty}$, satisfying a smallness condition. In \cite{NakSunUlm_1995}, the smallness condition 
was eliminated for smooth magnetic and electric potentials, and for compactly supported $C^2$ magnetic potentials and $L^\infty$ electric potentials.  The uniqueness results were subsequently extended to 
$C^1$ magnetic potentials in \cite{Tolmasky_1998}, to some less regular but small potentials in \cite{Panchenko_2002}, and to Dini continuous magnetic potentials in \cite{Salo_diss}. To the best of our knowledge, the latter result is the best one currently available, in terms of the regularity properties of magnetic potentials. 

The purpose of this paper is to extend the uniqueness result to the case of magnetic Schr\"odinger operators with magnetic potentials that are merely continuous. Our main result is as follows. 

\begin{thm}
\label{thm_main}

Let $\Omega\subset \R^n$, $n\ge 3$, be a bounded open set with $C^1$ boundary, and assume that $A_1,A_2\in C(\overline{\Omega},\C^n)$ and $q_1,q_2\in L^\infty(\Omega,\C)$. If $C_{A_1,q_1}=C_{A_2,q_2}$, then $dA_1=dA_2$ and $q_1=q_2$ in $\Omega$. 

\end{thm}

The key ingredient in the proof of Theorem \ref{thm_main} is a construction of  complex geometric optics solutions for the magnetic Schr\"odinger operator with a continuous magnetic potential.  
When constructing such solutions, we shall first derive a Carleman estimate for the magnetic Schr\"odinger operator $L_{A,q}$, with $A\in L^\infty(\Omega,\C^n)$ and $q\in L^\infty(\Omega,\C)$, with a gain of two derivatives, which  is based on the corresponding Carleman estimate for the Laplacian, obtained in \cite{Salo_Tzou_2009}.  To be able to proceed further with our construction, we also need to exploit a smoothing argument, for which it seems necessary to 
assume the continuity of $A$.  

Another important inverse boundary value problem, for which issues of the regularity have been studied extensively, is Calder\'on's problem for the conductivity equation, see \cite{Calderon}.  The unique identifiability of $C^2$ conductivities from boundary measurements was established in \cite{Syl_Uhl_1987}. The regularity assumptions were relaxed to conductivities having  $3/2+\varepsilon$ derivatives in \cite{Brown_1996}, and the uniqueness 
for conductivities having exactly $3/2$ derivatives was obtained in \cite{Paiv_Pan_Uhl}, see also \cite{Brown_Torres_2003}. In \cite{Green_Lassas_Uhlmann_2003}, uniqueness for conormal conductivities in $C^{1+\varepsilon}$ was shown.  The recent work \cite{Hab_Tataru} proves a uniqueness result for Calder\'on's problem with conductivities 
of class $C^1$ and with Lipschitz continuous conductivities, which are close to the identity in a suitable sense. 

The paper is organized as follows.   Section \ref{sec_CGO} contains the construction of complex geometric optics solutions for the magnetic Sch\"odinger operator with a continuous magnetic potential.  Section \ref{sec_proof_main} is devoted to the proof of Theorem \ref{thm_main}. In the proof it becomes essential to determine the values of the tangential component of the magnetic potential on the boundary.  This boundary determination result has been obtained in \cite{Brown_Salo_2006} for continuous magnetic potentials and $C^1$ domains, under the assumption that zero is not a Dirichlet eigenvalue for the magnetic Schr\"odinger operator in $\Omega$.  
Here we no longer assume that the Dirichlet problem is well posed, and to that end, in Section \ref{sec_boundary_rec}  we show how to adapt the method of \cite{Brown_Salo_2006} to the present general situation.

\section{Construction of complex geometric optics solutions} 

\label{sec_CGO}

Let $\Omega\subset\R^n$, $n\ge 3$, be a bounded open set.  Following \cite{DKSU_2007, KenSjUhl2007}, we shall use the method of Carleman estimates to construct complex geometric optics solutions for the magnetic Schr\"odinger equation $L_{A,q}u=0$ in $\Omega$,  with $A\in  C(\overline{\Omega},\C^n)$ and $q\in L^\infty(\Omega,\C)$.

Let us proceed by recalling a Carleman estimate for the semiclassical Laplace operator $-h^2\Delta$ with a gain of two derivatives, established in   \cite{Salo_Tzou_2009}, see also \cite{KenSjUhl2007}. 
Here $h>0$ is a small semiclassical parameter. 
Let $\tilde \Omega$ be an open set in $\R^n$ such that $ \Omega\subset\subset\tilde \Omega$ and 
$\varphi\in C^\infty(\tilde \Omega,\R)$.  Consider the conjugated operator 
\[
P_\varphi=e^{\frac{\varphi}{h}}(-h^2\Delta) e^{-\frac{\varphi}{h}}
\]
and its semiclassical principal symbol
\[
p_\varphi(x,\xi)=\xi^2+2i\nabla \varphi\cdot \xi-|\nabla \varphi|^2, \quad x\in \tilde {\Omega},\quad  \xi\in \R^n. 
\]
Following \cite{KenSjUhl2007}, we say that $\varphi$ is a limiting Carleman weight for $-h^2\Delta$ in $\tilde \Omega$, if $\nabla \varphi\ne 0$ in $\tilde \Omega$ and the Poisson bracket of $\Re p_\varphi$ and $\Im p_\varphi$ satisfies, 
\[
\{\Re p_\varphi,\Im p_\varphi\}(x,\xi)=0 \quad \textrm{when}\quad p_\varphi(x,\xi)=0, \quad (x,\xi)\in \tilde{\Omega}\times \R^n. 
\]
Examples of limiting Carleman weights  are  linear weights $\varphi(x)=\alpha\cdot x$, $\alpha\in \R^n$, $|\alpha|=1$, and logarithmic weights $\varphi(x)=\log|x-x_0|$,  with $x_0\not\in \tilde \Omega$.  In this paper we shall only use the linear weights. 

Our starting point is the following result due to \cite{Salo_Tzou_2009}.
\begin{prop}
Let $\varphi$ be a limiting Carleman weight for the semiclassical Laplacian on $\tilde \Omega$, and let $\varphi_\varepsilon=\varphi+\frac{h}{2\varepsilon}\varphi^2$.  Then 
for $0<h\ll \varepsilon\ll 1$ and $s\in\R$, we have
\begin{equation}
\label{eq_Carleman_lap}
\frac{h}{\sqrt{\varepsilon}}\|u\|_{H^{s+2}_{\emph{scl}}(\R^n)}\le C\|e^{\varphi_\varepsilon/h}(-h^2\Delta)e^{-\varphi_\varepsilon/h}u\|_{H^s_{\emph{scl}}(\R^n)},
\quad C>0,
\end{equation}
 for all $u\in C^\infty_0(\Omega)$.  
\end{prop}
Here 
\[
\|u\|_{H^s_{\textrm{scl}}(\R^n)}=\|\langle hD \rangle^s u\|_{L^2(\R^n)},\quad \langle\xi  \rangle=(1+|\xi|^2)^{1/2},
\]
is the natural semiclassical norm in the Sobolev space  $H^s(\R^n)$, $s\in\R$.

Next we shall derive a Carleman estimate for the magnetic Schr\"odinger operator $L_{A,q}$ with $A\in L^\infty(\Omega,\C^n)$ and $q\in L^\infty(\Omega,\C)$.  To that end we shall use the estimate \eqref{eq_Carleman_lap} with $s=-1$, and with $\varepsilon>0$ being sufficiently small but fixed, i.e. independent of $h$.  We have the following result. 
 
\begin{prop}
Let $\varphi\in C^\infty(\tilde\Omega,\R)$ be a limiting Carleman weight for the semiclassical Laplacian on $\tilde \Omega$, and assume that $A\in L^\infty(\Omega,\C^n)$ and $q\in L^\infty(\Omega,\C)$. Then 
for $0<h\ll 1$, we have 
\begin{equation}
\label{eq_Carleman_schr}
h\|u\|_{H^{1}_{\emph{scl}}(\R^n)}\le C\|e^{\varphi/h}(h^2L_{A,q})e^{-\varphi/h}u\|_{H^{-1}_{\emph{scl}}(\R^n)},
\end{equation}
 for all $u\in C^\infty_0(\Omega)$.  
\end{prop}

\begin{proof}
In order to prove the estimate \eqref{eq_Carleman_schr} it will be convenient to use the following characterization of the semiclassical norm in the Sobolev space $H^{-1}(\R^n)$, 
\begin{equation}
\label{eq_charac_H_-}
\|v\|_{H_{\textrm{scl}}^{-1}(\R^n)}=\sup_{0\ne \psi\in C^\infty_0(\R^n)}\frac{|\langle v,\psi\rangle_{\R^n} |}{\|\psi\|_{H_{\textrm{scl}}^{1}(\R^n)}},
\end{equation}
where $\langle \cdot,\cdot\rangle_{\R^n}$ is the distribution duality on $\R^n$. 

Let $\varphi_\varepsilon=\varphi+\frac{h}{2\varepsilon}\varphi^2$ be a convexified weight with $\varepsilon>0$ such that $0<h\ll \varepsilon\ll 1$, and let $u \in C^\infty_0(\Omega)$.  Then for all $0\ne \psi\in C^\infty_0(\R^n)$, we have 
\begin{align*}
|\langle e^{\varphi_\varepsilon/h} h^2 A\cdot D(e^{-\varphi_\varepsilon/h}u), \psi \rangle_{\R^n}|&\le 
\int_{\R^n}\bigg|hA\cdot \bigg(-u\bigg(1+\frac{h}{\varepsilon}\varphi\bigg)D\varphi+hDu\bigg)\psi\bigg|dx\\
&\le \mathcal{O}(h)\|u\|_{H^1_{\textrm{scl}}(\R^n)}\|\psi\|_{H^1_{\textrm{scl}}(\R^n)}.
\end{align*}
We also obtain that
\begin{align*}
|\langle e^{\varphi_\varepsilon/h} h^2 D\cdot( A e^{-\varphi_\varepsilon/h}u), \psi \rangle_{\R^n}|&\le \int_{\R^n}|h^2A e^{-\varphi_\varepsilon/h}u\cdot D(e^{\varphi_\varepsilon/h}\psi)|dx\\
&\le \mathcal{O}(h)\|u\|_{H^1_{\textrm{scl}}(\R^n)}\|\psi\|_{H^1_{\textrm{scl}}(\R^n)}.
\end{align*} 
Hence, using \eqref{eq_charac_H_-}, we get
\begin{equation}
\label{eq_2_3}
\|e^{\varphi_\varepsilon/h} h^2 A\cdot D(e^{-\varphi_\varepsilon/h}u) + e^{\varphi_\varepsilon/h} h^2 D\cdot( A e^{-\varphi_\varepsilon/h}u) \|_{H_{\textrm{scl}}^{-1}(\R^n)}\le \mathcal{O}(h)\|u\|_{H^1_{\textrm{scl}}(\R^n)}.
\end{equation}
Notice that the implicit constant in \eqref{eq_2_3} only depends on $\|A\|_{L^\infty(\Omega)}$, $\|\varphi\|_{L^\infty(\Omega)}$ and $\|D\varphi\|_{L^\infty(\Omega)}$.  Now choosing $\varepsilon >0$  sufficiently small but fixed, i.e. independent of $h$, we conclude from the estimate \eqref{eq_Carleman_lap} with $s=-1$ and the estimate \eqref{eq_2_3} that for all $h>0$ small enough, 
\begin{equation}
\label{eq_2_4}
\begin{aligned}
\|e^{\varphi_\varepsilon/h}(-h^2\Delta)e^{-\varphi_\varepsilon/h}u&+e^{\varphi_\varepsilon/h} h^2 A\cdot D(e^{-\varphi_\varepsilon/h}u) + e^{\varphi_\varepsilon/h} h^2 D\cdot( A e^{-\varphi_\varepsilon/h}u) \|_{H_{\textrm{scl}}^{-1}(\R^n)}\\
&\ge \frac{h}{C}\|u\|_{H^1_{\textrm{scl}}(\R^n)},\quad C>0. 
\end{aligned}
\end{equation}
Furthermore, the estimate
\[
\|h^2(A^2+q)u\|_{H^{-1}_{\textrm{scl}}(\R^n)}\le \mathcal{O}(h^2)\|u\|_{H^1_{\textrm{scl}}(\R^n)} 
\]
and the estimate \eqref{eq_2_4} imply that for all $h>0$ small enough, 
\[
\|e^{\varphi_\varepsilon/h}(h^2L_{A,q})e^{-\varphi_\varepsilon/h}u\|_{H_{\textrm{scl}}^{-1}(\R^n)}\ge \frac{h}{C}\|u\|_{H^1_{\textrm{scl}}(\R^n)},\quad C>0. 
\]
Using that 
\[
e^{-\varphi_\varepsilon/h}u=e^{-\varphi/h}e^{-\varphi^2/(2\varepsilon)}u,
\]
we obtain \eqref{eq_Carleman_schr}. The proof is complete.

\end{proof}

Let $\varphi\in C^\infty(\tilde\Omega,\R)$ be a limiting Carleman weight for $-h^2\Delta$.  
Then we have
\[
\langle L_\varphi u,\overline{v}\rangle_{\Omega}= \langle  u,\overline{ L_\varphi^* v}\rangle_{\Omega}, \quad u,v\in C^\infty_0(\Omega),
\]
where $L_\varphi=e^{\varphi/h}(h^2L_{A,q})e^{-\varphi/h}$ and $L_\varphi^*=e^{-\varphi/h}(h^2L_{\overline{A},\overline{q}})e^{\varphi/h}$ is the formal adjoint of $L_\varphi$. We have
\[
L_\varphi^*:C^\infty_0(\Omega)\to H^{-1}(\R^n)\cap\mathcal{E}'(\Omega)
\]
is bounded, and the estimate \eqref{eq_Carleman_schr} also holds for $L_\varphi^*$.

To construct complex geometric optics solutions for the magnetic Schr\"odinger operator we  need to convert the Carleman estimate  \eqref{eq_Carleman_schr} for $L_\varphi^*$ into 
the following solvability result. The proof is essentially well-known, and is included here for the convenience of the reader. We shall write
\begin{align*}
&\|u\|_{H^1_{\textrm{scl}}(\Omega)}^2=\|u\|_{L^2(\Omega)}^2+\|hDu\|_{L^2(\Omega)}^2,\\
&\|v\|_{H^{-1}_{\textrm{scl}}(\Omega)}=\sup_{0\ne \psi\in C_0^\infty(\Omega)}\frac{|\langle v,\psi\rangle_{\Omega}|}{\|\psi\|_{H^1_{\textrm{scl}}(\Omega)}}.
\end{align*}

\begin{prop}
\label{prop_solvability}
Let $A\in L^\infty(\Omega,\C^n)$, $q\in L^\infty(\Omega,\C)$, and let $\varphi$ be a limiting Carleman weight for the semiclassical Laplacian on $\tilde \Omega$. If $h>0$ is small enough, then for any $v\in H^{-1}(\Omega)$, there is a solution $u\in H^1(\Omega)$ of the equation
\[
e^{\varphi/h}(h^2L_{A,q})e^{-\varphi/h}u=v\quad\textrm{in}\quad \Omega,
\]
which satisfies
\[
\|u\|_{H^1_{\emph{scl}}(\Omega)}\le \frac{C}{h}\|v\|_{H^{-1}_{\emph{scl}}(\Omega)}.
\]
\end{prop}

\begin{proof}
Let $v\in H^{-1}(\Omega)$ and let us consider the following complex linear functional,
\[
L: L_\varphi^* C_0^\infty(\Omega)\to \C, \quad L_\varphi^* w \mapsto \langle w, \overline{v}\rangle_\Omega. 
\]
By the Carleman estimate \eqref{eq_Carleman_schr} for $L_\varphi^*$, the map $L$ is well-defined.  Let $w\in C_0^\infty(\Omega)$. Then we have
\begin{align*}
|L(L_\varphi^* w)|=|\langle w, \overline{v}\rangle_\Omega|&\le \|w\|_{H^1_{\textrm{scl}}(\R^n)}\|v\|_{H^{-1}_{\textrm{scl}}(\Omega)}\\
&\le \frac{C}{h}\|v\|_{H^{-1}_{\textrm{scl}}(\Omega)}\|L_\varphi^* w\|_{H^{-1}_{\textrm{scl}}(\R^n)}.
\end{align*}
By the Hahn-Banach theorem, we may extend $L$ to a linear continuous functional $\tilde L$ on $H^{-1}(\R^n)$, without increasing its norm. 
By the Riesz representation theorem, there exists $u\in H^1(\R^n)$ such that for all $\psi\in H^{-1}(\R^n)$,
\[
\tilde L(\psi)=\langle \psi,\overline{u}\rangle_{\R^n}, \quad \textrm{and}\quad \|u\|_{H^1_{\textrm{scl}}(\R^n)}\le \frac{C}{h}\|v\|_{H^{-1}_{\textrm{scl}}(\Omega)}. 
\]
Let us now show that $L_\varphi u=v$ in $\Omega$. To that end, let $w\in C_0^\infty(\Omega)$. Then 
\[
\langle L_\varphi u,\overline{w}\rangle_\Omega=\langle u,\overline{L_\varphi^*w}\rangle_{\R^n} =\overline{\tilde L(L_\varphi^*w)}=\overline{\langle w,\overline{v}\rangle_\Omega}=\langle v,\overline{w}\rangle_\Omega. 
\]
The proof is complete. 
\end{proof}

Our next goal is to construct complex geometric optics solutions for the magnetic Schr\"odinger equation 
\begin{equation}
\label{eq_2_6}
L_{A,q}u=0\quad \textrm{in} \quad \Omega,
\end{equation} 
with $A\in C(\overline{\Omega},\C^n)$ and $q\in L^\infty(\Omega,\C)$, using the solvability result of Proposition \ref{prop_solvability} and 
a smoothing argument. Complex geometric optics solutions are solutions of the form,
\begin{equation}
\label{eq_2_7}
u(x,\zeta;h)=e^{x\cdot\zeta/h} (a(x,\zeta;h)+r(x,\zeta;h)),
\end{equation}
where $\zeta\in\C^n$, $\zeta\cdot\zeta=0$, $|\zeta|\sim 1$,  $a$ is a smooth amplitude, $r$ is a correction term, and $h>0$ is a small parameter.

First, an application of the Tietze extension theorem allows us to extend $A\in C(\overline{\Omega},\C^n)$ to a continuous compactly supported vector field on the whole of $\R^n$. 
We consider the regularization $A^\sharp=A*\Psi_\tau\in C_0^\infty(\R^n,\C^n)$. Here $\tau>0$ is small and 
$\Psi_\tau(x)=\tau^{-n}\Psi(x/\tau)$ is the usual mollifier with $\Psi\in C^\infty_0(\R^n)$, $0\le \Psi\le 1$, and 
$\int \Psi dx=1$.  

We have the following estimates,
\begin{equation}
\label{eq_flat_est}
\|A-A^\sharp\|_{L^\infty(\R^n)}=o(1), \quad \tau \to 0,
\end{equation}
\begin{equation}
\label{eq_flat_est_2}
 \|\p^\alpha A^\sharp\|_{L^\infty(\R^n)}=\mathcal{O}(\tau^{-|\alpha|}),\quad \tau\to 0, \quad \textrm{for all}\quad \alpha,\quad |\alpha|\ge 0.
\end{equation}

It will be convenient to introduce the following bounded operator, 
\[
m_A: H^1(\Omega)\to H^{-1}(\Omega),\quad  m_A(u)=D\cdot(A u),
\]
where the distribution $m_A(u)$ is given by
\[
\langle m_A(u),v \rangle_\Omega=-\int_{\Omega} Au\cdot Dv dx, \quad v\in C_0^\infty(\Omega). 
\]

Let us conjugate $h^2L_{A,q}$ by $e^{x\cdot\zeta/h}$. First, let us compute $e^{-x\cdot\zeta/h}\circ h^2m_A \circ e^{x\cdot \zeta/h}$. When $u\in H^1(\Omega)$ and $v\in C^\infty_0(\Omega)$, we get 
\begin{align*}
\langle e^{-x\cdot\zeta/h}h^2m_A (e^{x\cdot \zeta/h}u),v\rangle_\Omega &=-\int_{\Omega} h^2 A e^{x\cdot \zeta/h}u \cdot D(e^{-x\cdot\zeta/h} v)dx\\
&=-\int_\Omega(hi\zeta\cdot Auv+h^2Au\cdot Dv)dx,
\end{align*}
and therefore, 
\[
e^{-x\cdot\zeta/h}\circ h^2m_A \circ e^{x\cdot \zeta/h}= -hi\zeta\cdot A+h^2 m_A. 
\]
Furthermore, we obtain that 
\begin{align*}
e^{-x\cdot\zeta/h}\circ (-h^2\Delta) \circ e^{x\cdot \zeta/h}=-h^2\Delta -2ih \zeta\cdot D,\\
e^{-x\cdot\zeta/h}\circ h^2(A\cdot D)\circ  e^{x\cdot \zeta/h}=h^2A\cdot D- hi\zeta\cdot A. 
\end{align*}
Hence, we have
\begin{equation}
\label{eq_2_8}
e^{-x\cdot\zeta/h} h^2 \circ L_{A,q} \circ e^{x\cdot \zeta/h}=-h^2\Delta -2ih \zeta\cdot D +h^2A\cdot D- 2 hi\zeta\cdot A+ h^2 m_A + h^2(A^2+q). 
\end{equation}

In this paper we shall work with $\zeta$ depending slightly on $h$, i.e. $\zeta=\zeta_0+\zeta_1$ with $\zeta_0$ being independent of $h$ and $\zeta_1=\mathcal{O}(h)$ as $h\to 0$.  We also assume that $|\Re\zeta_0|=|\Im\zeta_0|=1$. Then it will be convenient to write \eqref{eq_2_8} in the following form,
\begin{align*}
e^{-x\cdot\zeta/h} h^2 \circ L_{A,q}\circ  e^{x\cdot \zeta/h}= &  -h^2\Delta -2ih \zeta_0\cdot D -2ih \zeta_1\cdot D +h^2A\cdot D- 2 hi\zeta_0\cdot A^\sharp\\
&- 2 hi\zeta_0\cdot (A-A^\sharp) - 2 hi\zeta_1\cdot A+ h^2 m_A + h^2(A^2+q). 
\end{align*}

In order that \eqref{eq_2_7} be a solution of \eqref{eq_2_6}, we require that 
\begin{equation}
\label{eq_2_9}
\zeta_0\cdot Da +\zeta_0\cdot A^\sharp a=0\quad\textrm{in}\quad \R^n,
\end{equation}
and 
\begin{equation}
\label{eq_2_10}
\begin{aligned}
e^{-x\cdot\zeta/h} h^2 L_{A,q}& e^{x\cdot \zeta/h}r= -(- h^2\Delta a  +h^2A\cdot D a + h^2 m_A (a)+ h^2(A^2+q)a)\\
& +2ih \zeta_1\cdot Da  + 2 hi\zeta_0\cdot (A-A^\sharp)a +2 hi\zeta_1\cdot Aa=:g \quad\textrm{in}\quad \Omega. 
\end{aligned}
\end{equation}
The equation \eqref{eq_2_9} is the first transport equation and  one looks for its solution in the form $a=e^{\Phi^\sharp}$, where $\Phi^\sharp$ solves the equation
\begin{equation}
\label{eq_2_9_phi}
\zeta_0\cdot\nabla \Phi^\sharp + i\zeta_0\cdot A^\sharp=0\quad \textrm{in}\quad \R^n. 
\end{equation}
As $\zeta_0\cdot\zeta_0=0$ and $|\Re\zeta_0|=|\Im\zeta_0|=1$,  the operator   $N_{\zeta_0}:=\zeta_0\cdot\nabla$ is the $\bar\p$--operator in suitable linear coordinates. Let us introduce an inverse operator defined by
\[
(N_{\zeta_0}^{-1}f)(x) =\frac{1}{2\pi}\int_{\R^2} \frac{f(x-y_1\Re\zeta_0 -y_2\Im\zeta_0)}{y_1+iy_2}dy_1dy_2,\quad f\in C_0(\R^n).
\]
We have the following result, see \cite[Lemma 4.6]{Salo_diss}. 
\begin{lem}
\label{lem_salo_1}
 Let $f\in W^{k,\infty}(\R^n)$, $k\ge 0$, with $\emph{\supp}(f)\subset B(0,R)$. Then $\Phi=N_{\zeta_0}^{-1} f\in W^{k,\infty}(\R^n)$ satisfies $N_{\zeta_0}\Phi=f$ in $\R^n$, and we have 
\begin{equation}
\label{eq_salo_1}
\|\Phi\|_{W^{k,\infty}(\R^n)}\le C\|f\|_{W^{k,\infty}(\R^n)}, 
\end{equation}
where $C=C(R)$. If $f\in C_0(\R^n)$, then $\Phi\in C(\R^n)$. 
\end{lem}
Thanks to Lemma \ref{lem_salo_1},  the function $\Phi^\sharp(x,\zeta_0; \tau)=N_{\zeta_0}^{-1}(-i\zeta_0\cdot A^\sharp)\in C^\infty(\R^n)$ satisfies the equation \eqref{eq_2_9_phi}. 
Furthermore, the estimate \eqref{eq_flat_est_2} implies that 
\begin{equation}
\label{eq_ampl_est}
\|\p^\alpha \Phi^\sharp \|_{L^\infty(\R^n)}\le C_\alpha \tau^{-|\alpha|}, \quad \textrm{for all}\quad \alpha, \quad |\alpha|\ge 0.   
\end{equation}
It follows from \eqref{eq_flat_est} and \eqref{eq_salo_1} that $\Phi^\sharp (x,\zeta_0;\tau)$ converges uniformly in $\R^n$ to $\Phi (x,\zeta_0):=N_{\zeta_0}^{-1}(-i\zeta_0\cdot A)\in C(\R^n)$ as $\tau\to 0$.

Let us turn now to the equation \eqref{eq_2_10}. First notice that the right hand side $g$Ê of \eqref{eq_2_10} belongs to $H^{-1}(\Omega)$ and we would like to  estimate $\|g\|_{H^{-1}_{\textrm{scl}}(\Omega)}$.  To that end,  let $0\ne \psi\in C_0^\infty(\Omega)$. Then using \eqref{eq_ampl_est} and the fact that $\zeta_1=\mathcal{O}(h)$, we get
\begin{align*}
&|\langle h^2\Delta a, \psi \rangle_\Omega|\le \mathcal{O} (h^2/\tau^{2}) \|\psi\|_{L^2(\Omega)}\le \mathcal{O} (h^2/\tau^{2}) \|\psi\|_{H^{1}_{\textrm{scl}}(\Omega)}, \\
&|\langle h^2 A\cdot Da,\psi\rangle_\Omega|\le \mathcal{O} (h^2/\tau) \|\psi\|_{H^{1}_{\textrm{scl}}(\Omega)},\\
& |\langle 2ih\zeta_1\cdot Da,\psi\rangle_\Omega|\le \mathcal{O} (h^2/\tau) \|\psi\|_{H^{1}_{\textrm{scl}}(\Omega)},\\
& |\langle 2hi\zeta_1\cdot Aa,\psi \rangle_\Omega|\le \mathcal{O} (h^2) \|\psi\|_{H^{1}_{\textrm{scl}}(\Omega)}.
\end{align*}
With the help \eqref{eq_flat_est}, \eqref{eq_flat_est_2}, and  \eqref{eq_ampl_est}, we obtain that 
\begin{align*}
|\langle h^2m_A(a),& \psi \rangle_\Omega|\le \bigg| \int_\Omega h^2 A^\sharp a\cdot D\psi dx\bigg| + \bigg|\int_\Omega h^2 (A-A^\sharp) a\cdot D\psi dx\bigg|\\
&\le \bigg| \int_\Omega h^2 D\cdot (A^\sharp a) \cdot \psi dx\bigg| +\mathcal{O}(h)\|A-A^\sharp\|_{L^\infty}\|a\|_{L^\infty}\|hD\psi\|_{L^2(\Omega)}\\
&\le (\mathcal{O}(h^2/\tau) +\mathcal{O}(h)o_{\tau\to 0}(1)) \|\psi\|_{H^1_{\textrm{scl}}(\Omega)}. 
\end{align*}
Using \eqref{eq_flat_est} and  \eqref{eq_ampl_est}, we have
\[
|\langle 2hi\zeta_0 \cdot(A-A^\sharp)a,\psi\rangle_\Omega |\le \mathcal{O}(h)o_{\tau\to 0}(1) \|\psi\|_{H^1_{\textrm{scl}}(\Omega)}.
\]
Combining the estimates above and the estimate $\|h^2(A^2+q)a\|_{L^2(\Omega)}\le \mathcal{O}(h^2)$, we conclude that 
\[
\|g\|_{H^{-1}_{\textrm{scl}}(\Omega)}\le  \mathcal{O} (h^2/\tau^{2}) + \mathcal{O}(h)o_{\tau\to 0}(1).
\]
Choosing now $\tau=h^\sigma$ with $0<\sigma<1/2$, we get
\begin{equation}
\label{eq_2_11}
\|g\|_{H^{-1}_{\textrm{scl}}(\Omega)}=o(h) \quad \textrm{as}\quad h\to 0. 
\end{equation}
Thanks to Proposition \ref{prop_solvability} and \eqref{eq_2_11}, for $h>0$ small enough, there exists a solution $r\in H^1(\Omega)$ of \eqref{eq_2_10} such that $\|r\|_{H^1_{\textrm{scl}}(\Omega)}=o(1)$ as $h\to 0$.

Summing up, we have proved the following result. 
\begin{prop}
\label{prop_cgo_solutions}
Let $\Omega\subset \R^n$, $n\ge 3$, be a bounded open set.  Let $A\in C(\overline{\Omega},\C^n)$, $q\in L^\infty(\Omega,\C)$, and let $\zeta\in \C^n$ be such that $\zeta\cdot\zeta=0$, $\zeta=\zeta_0+\zeta_1$ with $\zeta_0$ being independent of $h>0$, $|\emph{\Re}\zeta_0|=|\emph{\Im} \zeta_0|=1$, and $\zeta_1=\mathcal{O}(h)$ as $h\to 0$.    Then for all $h>0$ small enough, there exists a solution $u(x,\zeta;h)\in H^1(\Omega)$ to the magnetic Schr\"odinger equation $L_{A,q}u=0$ in $\Omega$, of the form
\[
u(x,\zeta;h)=e^{x\cdot\zeta/h}(e^{\Phi^\sharp(x,\zeta_0;h)}+r(x,\zeta;h)).
\] 
The function  $\Phi^\sharp (\cdot,\zeta_0;h)\in C^\infty(\R^n)$ satisfies 
$\|\p^\alpha \Phi^\sharp\|_{L^\infty(\R^n)}\le C_\alpha h^{-\sigma|\alpha|}$, $0<\sigma<1/2$,   
for all $\alpha$, $|\alpha|\ge 0$, and $\Phi^\sharp (x,\zeta_0;h)$ converges uniformly in $\R^n$ to $\Phi(\cdot,\zeta_0):=N_{\zeta_0}^{-1}(-i\zeta_0\cdot A)\in C(\R^n)$ as $h\to 0$. Here we have extended $A$ to a  $C_0(\R^n,\C^n)$-vector field. The remainder $r$ is such that $\|r\|_{H^1_{\emph{\textrm{scl}}}(\Omega)}=o(1)$ as $h\to 0$. 
\end{prop}

\section{Proof of Theorem \ref{thm_main}}

\label{sec_proof_main}

The first step is the derivation of an integral identity based on the fact that $C_{A_1,q_1}=C_{A_2,q_2}$, see also \cite[Lemma 4.3]{Salo_diss}. 

\begin{prop}
\label{prop_eq_int_identity}
Let $\Omega\subset \R^n$, $n\ge 3$,  be a bounded open set with $C^1$ boundary.  Assume that $A_1,A_2\in L^\infty(\Omega,\C^n)$ and $q_1,q_2\in L^\infty(\Omega,\C)$. If $C_{A_1,q_1}=C_{A_2,q_2}$, then the following integral identity 
\begin{equation}
\label{eq_int_identity}
\int_\Omega i(A_1-A_2)\cdot (u_1\nabla \overline{u_2}-\overline{u_2}\nabla u_1)dx+\int_\Omega(A_1^2-A_2^2+q_1-q_2)u_1\overline{u_2}dx=0 
\end{equation}
holds for any $u_1,u_2\in H^1(\Omega)$ satisfying $L_{A_1,q_1}u_1=0$ in $\Omega$ and $L_{\overline{A_2},\overline{q_2}}u_2=0$ in $\Omega$, respectively.  
\end{prop}

\begin{proof}
Let $u_1, u_2\in H^1(\Omega)$ be  solutions to $L_{A_1,q_1}u_1=0$ in $\Omega$ and 
\begin{equation}
\label{eq_u_2_solves}
L_{\overline{A_2},\overline{q_2}}u_2=0\quad \textrm{in}\quad \Omega, 
\end{equation}
respectively.  Then the fact that $C_{A_1,q_1}=C_{A_2,q_2}$ implies that there is $v_2\in H^1(\Omega)$ satisfying $L_{A_2,q_2}v_2=0$ in $\Omega$ such that 
\[
u_1=v_2\quad\textrm{and}\quad \p_\nu u_1+i(A_1\cdot\nu)u_1=\p_\nu v_2+i(A_2\cdot \nu)v_2\quad\textrm{on}\quad \p\Omega. 
\]
In particular, 
\begin{equation}
\label{eq_3_1}
\langle \p_\nu u_1+i(A_1\cdot\nu)u_1,\overline{u_2} \rangle_{\p \Omega}=\langle \p_\nu v_2+i(A_2\cdot \nu)v_2, \overline{u_2} \rangle_{\p \Omega}.
\end{equation}

It follows from \eqref{eq_trace} and \eqref{eq_u_2_solves} that 
\begin{equation}
\label{eq_3_2}
\langle \p_\nu \overline{u_2}-i(A_2\cdot\nu)\overline{u_2},\overline{g} \rangle_{\p \Omega}=\int_{\Omega}( \nabla \overline{u_2}\cdot\nabla\overline{g}-iA_2\cdot(\overline{u_2}\nabla\overline{g}-\overline{g}\nabla\overline{u_2})+(A_2^2+q_2)\overline{u_2}\,\overline{g})dx,
\end{equation}
for any $g\in H^1(\Omega)$. 
Using \eqref{eq_trace} and  \eqref{eq_3_2}, we obtain that 
\begin{equation}
\label{eq_3_3}
\langle \p_\nu v_2+i(A_2\cdot \nu)v_2, \overline{u_2} \rangle_{\p \Omega}=\langle \p_\nu \overline{u_2}-i(A_2\cdot\nu)\overline{u_2},v_2 \rangle_{\p \Omega}=\langle \p_\nu \overline{u_2}-i(A_2\cdot\nu)\overline{u_2},u_1 \rangle_{\p \Omega}.
\end{equation}
In the last equality we have also used that $u_1=v_2$ on $\p \Omega$. 
Combining \eqref{eq_3_1} and \eqref{eq_3_3}, we get
\[
\langle \p_\nu u_1+i(A_1\cdot\nu)u_1,\overline{u_2} \rangle_{\p \Omega}=\langle \p_\nu \overline{u_2}-i(A_2\cdot\nu)\overline{u_2},u_1 \rangle_{\p \Omega},
\]
which shows the identity \eqref{eq_int_identity}. The proof is complete. 
\end{proof}

The second step is the determination of the boundary values of the tangential components of the magnetic potentials.  First of all,  we may assume without loss of generality that the normal component of $A_j$ satisfies $A_j\cdot\nu=0$ on $\p \Omega$, $j=1,2$. Indeed, as $\p \Omega\in C^1$, by \cite[Theorem 1.3.3]{Horm_book_1}, there is $\psi_j\in C^{1}(\overline{\Omega})$ which satisfies $\psi_j=0$ on $\p \Omega$ and $(A_j+  \nabla \psi_j)\cdot\nu =0$ on $\p \Omega$, $j=1,2$.  Hence, it  follows from \eqref{eq_invariance_cauchy} that $ A_j$ can be replaced by $A_j+\nabla \psi_j$, $j=1,2$.  Furthermore,  by Proposition \ref{prop_boundary_rec} in Section \ref{sec_boundary_rec}
 applied to each connected component of $\Omega$, we conclude 
that $A_1=A_2$ on $\p \Omega$.
This allows us to extend $A_1$ and $A_2$ to $\R^n$ so that the extensions, which we shall denote by the same letters, agree on $\R^n\setminus\Omega$, and satisfy $A_1,A_2\in C_0(\R^n,\C^n)$,  $j=1,2$.  We also extend $q_j$, $j=1,2$, by zero to $\R^n\setminus\Omega$.

The next step is to use the integral identity \eqref{eq_int_identity} with $u_1$ and $u_2$ being complex geometric optics solutions for the magnetic Schr\"odinger equations in $\Omega$. To construct such solutions,  let $\xi,\mu_1,\mu_2\in\R^n$ be such that $|\mu_1|=|\mu_2|=1$ and $\mu_1\cdot\mu_2=\mu_1\cdot\xi=\mu_2\cdot\xi=0$. 
Similarly to \cite{Sun_1993}, we set 
\begin{equation}
\label{eq_zeta_1_2}
\zeta_1=\frac{ih\xi}{2}+\mu_1 + i\sqrt{1-h^2\frac{|\xi|^2}{4}}\mu_2 , \quad 
\zeta_2=-\frac{ih\xi}{2}-\mu_1+i\sqrt{1-h^2\frac{|\xi|^2}{4}}\mu_2,
\end{equation}
so that $\zeta_j\cdot\zeta_j=0$, $j=1,2$, and $(\zeta_1+\overline{\zeta_2})/h=i\xi$. Here $h>0$ is a small enough semiclassical parameter.  Moreover, $\zeta_1= \mu_1+ i\mu_2+\mathcal{O}(h)$ and $\zeta_2= -\mu_1+ i\mu_2+\mathcal{O}(h)$ as $h\to 0$. 

By Proposition \ref{prop_cgo_solutions},  for all $h>0$ small enough, there exists a solution $u_1(x,\zeta_1;h)\in H^1(\Omega)$ to the magnetic Schr\"odinger equation $L_{A_1,q_1}u_1=0$ in $\Omega$, of the form
\begin{equation}
\label{eq_u_1}
u_1(x,\zeta_1;h)=e^{x\cdot\zeta_1/h}(e^{\Phi_1^\sharp(x,\mu_1+i\mu_2;h)}+r_1(x,\zeta_1;h)),
\end{equation}
where $\Phi_1^\sharp(\cdot,\mu_1+i\mu_2;h)\in C^\infty(\R^n)$ satisfies the estimate 
\begin{equation}
\label{eq_est_phi_1}
\|\p^\alpha \Phi_1^\sharp\|_{L^\infty(\R^n)}\le C_\alpha h^{-\sigma|\alpha|},\quad 0<\sigma<1/2,
\end{equation}
for all $\alpha$, $|\alpha|\ge 0$, $\Phi_1^\sharp(x,\mu_1+i\mu_2;h)$ converges uniformly in $\R^n$ to $\Phi_{1}(x,\mu_1+i\mu_2):=N_{\mu_1+i\mu_2}^{-1}
(-i(\mu_1+i\mu_2)\cdot A_1)\in C(\R^n)$ as $h\to 0$, 
and 
\begin{equation}
\label{eq_est_r_1}
\|r_1\|_{H^1_{\textrm{scl}}(\Omega)}=o(1)\quad \textrm{as}\quad h\to 0.
\end{equation}
Similarly, for all $h>0$ small enough, there exists a solution $u_2(x,\zeta_2;h)\in H^1(\Omega)$ to the magnetic Schr\"odinger equation $L_{\overline{A_2},\overline{q_2}}u_2=0$ in $\Omega$, of the form
\begin{equation}
\label{eq_u_2}
u_2(x,\zeta_2;h)=e^{x\cdot\zeta_2/h}(e^{\Phi_2^\sharp(x,-\mu_1+i\mu_2;h)}+r_2(x,\zeta_2;h)),
\end{equation}
where $\Phi_2^\sharp(\cdot,-\mu_1+i\mu_2;h)\in C^\infty(\R^n)$ satisfies the estimate 
\begin{equation}
\label{eq_est_phi_2}
\|\p^\alpha \Phi_2^\sharp\|_{L^\infty(\R^n)}\le C_\alpha h^{-\sigma|\alpha|},\quad 0<\sigma<1/2,
\end{equation}
for all $\alpha$, $|\alpha|\ge 0$, $\Phi_2^\sharp(x,-\mu_1+i\mu_2;h)$ converges uniformly in $\R^n$ to $\Phi_{2}(x,-\mu_1+i\mu_2):=N_{-\mu_1+i\mu_2}^{-1}
(-i(-\mu_1+i\mu_2)\cdot \overline{A_2})\in C(\R^n)$ as $h\to 0$, 
and 
\begin{equation}
\label{eq_est_r_2}
\|r_2\|_{H^1_{\textrm{scl}}(\Omega)}=o(1)\quad \textrm{as}\quad h\to 0.
\end{equation}

Next we shall substitute $u_1$ and $u_2$, given by \eqref{eq_u_1} and \eqref{eq_u_2}, into the integral identity \eqref{eq_int_identity}, multiply it by $h$, and let $h\to 0$. To that end, we first compute
\begin{align*}
h u_1\nabla\overline{u_2}=&\overline{\zeta_2}e^{ix\cdot\xi}(e^{\Phi_1^\sharp+\overline{\Phi_2^\sharp}}+e^{\Phi_1^\sharp}\overline{r_2}+r_1e^{\overline{\Phi_2^\sharp}}+r_1\overline{r_2})\\
&+he^{ix\cdot\xi}(e^{\Phi_1^\sharp}\nabla e^{\overline{\Phi_2^\sharp}} + e^{\Phi_1^\sharp}\nabla \overline{r_2} + r_1\nabla e^{\overline{\Phi_2^\sharp}} +r_1  \nabla \overline{r_2}).
\end{align*}
Using the estimates \eqref{eq_est_phi_1}, \eqref{eq_est_r_1}, \eqref{eq_est_phi_2} and \eqref{eq_est_r_2}, we obtain that 
\begin{align*}
|\int_\Omega & i(A_1-A_2)\cdot \overline{\zeta_2} e^{ix\cdot\xi}(e^{\Phi_1^\sharp}\overline{r_2}+r_1e^{\overline{\Phi_2^\sharp}}+r_1\overline{r_2})dx|\\
&\le C\|A_1-A_2\|_{L^\infty}
(\|e^{\Phi_1^\sharp}\|_{L^2}\|\overline{r_2}\|_{L^2}+\|r_1\|_{L^2}\|e^{\overline{\Phi_2^\sharp}}\|_{L^2}+\|r_1\|_{L^2}\|\overline{r_2}\|_{L^2})=o(1),
\end{align*}
as $h\to 0$. Furthermore, 
\begin{align*}
|\int_\Omega h i(A_1-A_2)\cdot e^{ix\cdot\xi}(e^{\Phi_1^\sharp}\nabla e^{\overline{\Phi_2^\sharp}} + e^{\Phi_1^\sharp}\nabla \overline{r_2} + r_1\nabla e^{\overline{\Phi_2^\sharp}} +r_1  \nabla \overline{r_2})dx|\\
\le \mathcal{O}(h)(h^{-\sigma}+ h^{-1}o(1)+ o(1)h^{-\sigma}+o(1)h^{-1})=o(1),
\end{align*}
as $h\to 0$. Here $0<\sigma<1/2$.   We also have
\begin{align*}
|h\int_\Omega(A_1^2-A_2^2+q_1-q_2)e^{ix\cdot\xi}(e^{\Phi_1^\sharp+\overline{\Phi_2^\sharp}}+e^{\Phi_1^\sharp}\overline{r_2}+r_1e^{\overline{\Phi_2^\sharp}}+r_1\overline{r_2})dx|=\mathcal{O}(h),
\end{align*}
as $h\to 0$. Hence, substituting  $u_1$ and $u_2$, given by \eqref{eq_u_1} and \eqref{eq_u_2}, into the integral identity \eqref{eq_int_identity}, multiplying it by $h$, and letting $h\to 0$, we get
\begin{equation}
\label{eq_with_phases}
(\mu_1+i\mu_2)\cdot\int_\Omega (A_1-A_2) e^{ix\cdot\xi} e^{\Phi_{1}(x,\mu_1+i\mu_2)+\overline{\Phi_{2}(x,-\mu_1+i\mu_2)}}dx=0,
\end{equation}
where 
\begin{align*}
\Phi_{1}=N_{\mu_1+i\mu_2}^{-1}(- i(\mu_1+i\mu_2)\cdot A_1) \in C(\R^n),\\
\Phi_{2}= N_{-\mu_1+i\mu_2}^{-1} ( - i(-\mu_1+i\mu_2)\cdot \overline{A_2})\in C(\R^n).
\end{align*}

Recalling that $A_1=A_2$ on $\R^n\setminus{\Omega}$, we may replace the integration in \eqref{eq_with_phases} over $\Omega$ by an integration over all of $\R^n$ so that 
\begin{equation}
\label{eq_with_phases_R_n}
(\mu_1+i\mu_2)\cdot\int_{\R^n} (A_1-A_2) e^{ix\cdot\xi} e^{\Phi_{1}(x,\mu_1+i\mu_2)+\overline{\Phi_{2}(x,-\mu_1+i\mu_2)}}dx=0. 
\end{equation}

The next step is to remove the  function $e^{\Phi_{1}+\overline{\Phi_{2}}}$ in the integral \eqref{eq_with_phases_R_n}. First using the following properties of the Cauchy transform, 
\[
\overline{N_{\zeta}^{-1}f}=N_{\overline{\zeta}}^{-1}\overline{f},\quad  N_{-\zeta}^{-1}f=-N_\zeta^{-1}f,
\]
we notice that 
\begin{equation}
\label{eq_sum_phase}
\Phi_{1}+\overline{\Phi_{2}}=N_{\mu_1+i\mu_2}^{-1}(- i(\mu_1+i\mu_2)\cdot (A_1-A_2)).  
\end{equation}

We shall next need the following result, which is due to \cite{Eskin_Ralston_1995}, see also \cite{Salo_2006}.  
\begin{prop} 
\label{prop_eskin_ralston}
Let  $\xi,\mu_1,\mu_2\in\R^n$ be such that $|\mu_1|=|\mu_2|=1$ and $\mu_1\cdot\mu_2=\mu_1\cdot\xi=\mu_2\cdot\xi=0$. 
Let $W\in C_0(\R^n, \C^n)$ and $\phi=N_{\mu_1+i\mu_2}^{-1}(- i(\mu_1+i\mu_2)\cdot W)$. Then 
\begin{equation}
\label{eq_eskin_ralston}
(\mu_1+i\mu_2)\cdot\int_{\R^n} W(x) e^{ix\cdot\xi} e^{\phi(x)} dx=(\mu_1+i\mu_2)\cdot\int_{\R^n} W(x) e^{ix\cdot\xi} dx.
\end{equation}
\end{prop}

\begin{proof}
By a standard approximation argument, it suffices to prove \eqref{eq_eskin_ralston} for $W\in C^\infty_0(\R^n, \C^n)$ and 
\begin{equation}
\label{eq_3_15}
\phi=N_{\mu_1+i\mu_2}^{-1}(- i(\mu_1+i\mu_2)\cdot W)\in C^\infty(\R^n).
\end{equation}

We can always assume that $\mu_1=(1,0,\dots, 0)$ and $\mu_2=(0,1,0,\dots, 0)$, so that $\xi=(0,0,\xi'')$, $\xi''\in\R^{n-2}$, and therefore,
\[
(\p_{x_1}+i\p_{x_2})\phi=-i(\mu_1+i\mu_2)\cdot W\quad \textrm{in}\quad \R^n.
\]
Hence, writing $x=(x',x'')$, $x'=(x_1,x_2)$, $x''\in \R^{n-2}$, we get 
\begin{align*}
(\mu_1+i\mu_2)\cdot\int_{\R^n} W(x) e^{ix\cdot\xi} e^{\phi(x)} dx&=i\int_{\R^n}e^{ix''\cdot\xi''}e^{\phi(x)}(\p_{x_1}+i\p_{x_2})\phi(x) dx\\
&=i\int_{\R^{n-2}}e^{ix''\cdot\xi''} h(x'')dx'',
\end{align*}
where 
\begin{align*}
h(x'')=\int_{\R^2} (\p_{x_1}+i\p_{x_2})e^{\phi(x)} dx'&=\lim_{R\to \infty} \int_{|x'|\le R} (\p_{x_1}+i\p_{x_2})e^{\phi(x)} dx'\\
&=\lim_{R\to \infty} \int_{|x'|= R}e^{\phi(x)}(\nu_1+i\nu_2)dS_R(x').
\end{align*}
Here $\nu=(\nu_1,\nu_2)$ is the unit outer normal to the circle $|x'|= R$, and we have used the Gauss theorem.  

It follows from \eqref{eq_3_15} that $|\phi(x',x'')|=\mathcal{O}(1/|x'|)$ as $|x'|\to \infty$.  Hence, we have 
\[
e^\phi=1+\phi+\mathcal{O}(|\phi|^2)=1+\phi+\mathcal{O}(|x'|^{-2})\quad \textrm{as}\quad  |x'|\to \infty. 
\]
Since
\begin{align*}
\int_{|x'|= R}(\nu_1+i\nu_2)dS_R(x')=\int_{|x'|\le R}(\p_{x_1}+\p_{x_2}) (1)dx'=0,\\
\bigg|Ê\int_{|x'|= R}\mathcal{O}(|x'|^{-2})(\nu_1+i\nu_2)dS_R(x')\bigg|\le \mathcal{O}(R^{-1})\quad \textrm{as}\quad R\to \infty, 
\end{align*}
we obtain that 
\begin{align*}
h(x'')=\lim_{R\to \infty} \int_{|x'|= R}\phi(x)(\nu_1+i\nu_2)dS_R(x')&=\lim_{R\to \infty} \int_{|x'|\le  R} (\p_{x_1}+i\p_{x_2})\phi(x)dx'\\
&=-\int_{\R^2} i(\mu_1+i\mu_2)\cdot W(x)dx',
\end{align*}
which shows the claim. The proof is complete. 

\end{proof}

Combining Proposition \ref{prop_eskin_ralston} with \eqref{eq_with_phases_R_n} and \eqref{eq_sum_phase}, we get 
\begin{equation}
\label{eq_without_phases}
(\mu_1+i\mu_2)\cdot\int_{\R^n} (A_1(x)-A_2(x)) e^{ix\cdot\xi} dx=0. 
\end{equation}

The next step is to derive from \eqref{eq_without_phases} that $dA_1=dA_2$ in $\R^n$. While this step is well-known, see \cite[Lemma 4.8]{Salo_diss}, we present it here for completeness and for convenience of the reader.  

\begin{prop}
\label{prop_curl_w}
Let $W\in C_0(\R^n,\C^n)$. Assume that 
\begin{equation}
\label{eq_without_phases_2}
(\mu_1+i\mu_2)\cdot\int_{\R^n} W(x) e^{ix\cdot\xi} dx=0,
\end{equation}
whenever $\xi,\mu_1,\mu_2\in\R^n$ are such that $|\mu_1|=|\mu_2|=1$ and $\mu_1\cdot\mu_2=\mu_1\cdot\xi=\mu_2\cdot\xi=0$. Then $dW=0$ in $\R^n$. 
\end{prop} 

\begin{proof}
It follows from \eqref{eq_without_phases_2} that 
\begin{equation}
\label{eq_without_phases_3}
\mu\cdot \hat W(\xi)=0
\end{equation}
whenever $\mu\in\R^n$ is such that $\mu\cdot\xi=0$. Here $\hat W$ is the Fourier transform of $W$. 
Let $\mu_{jk}(\xi)=\xi_j e_k-\xi_k e_j$ for $j\ne k$.  Then $\mu_{jk}(\xi)\cdot\xi=0$, and therefore, \eqref{eq_without_phases_3} implies that 
\[
\xi_j\hat W_k(\xi)-\xi_k\hat W_j(\xi)=0. 
\]
Hence, $\p_{x_j}W_k-\p_{x_k}W_j=0$ in $\R^n$ in the sense of distributions,  for $j\ne k$. The proof is complete. 

\end{proof}

By Proposition \ref{prop_curl_w} we conclude from \eqref{eq_without_phases} that $dA_1=dA_2$ in $\R^n$. Therefore,  $A_1-A_2=\nabla \psi$ in $\R^n$, with some $\psi\in C^1_0(\R^n)$. 

Next we would like to show that $q_1=q_2$ in $\Omega$. To that end, we first want to add $\nabla \psi$ to the potential $A_2$ without changing the set of the Cauchy data. 
When doing so, it will be convenient to work on a large open ball $B$ such that $\Omega\subset\subset B$ and $\supp(\psi)\subset B$. 
We shall need the following result, see \cite[Lemma 4.2]{Salo_diss}.

\begin{prop}
\label{prop_Cauchy_data}
Let $\Omega\subset \R^n$ be a bounded open set with $C^1$ boundary,  and let $B$ be a open ball in $\R^n$ such that $\Omega\subset\subset B$.  Let $A_1,A_2\in L^\infty(B,\C^n)$, and  $q_1,q_2\in L^\infty(B,\C)$. Assume that 
\[
A_1=A_2\quad\textrm{and}\quad q_1=q_2\quad \textrm{in}\quad B\setminus\Omega.
\]
If  $C_{A_1,q_1}=C_{A_2,q_2}$ then $C'_{A_1,q_1}=C'_{A_2,q_2}$, where $C'_{A_j,q_j}$ is the set of the Cauchy data for $L_{A_j,q_j}$ in $B$, $j=1,2$. 
\end{prop}

\begin{proof}
Let $u_1'\in H^1(B)$ be a solution to $L_{A_1,q_1}u_1'=0$ in $B$. As $C_{A_1,q_1}=C_{A_2,q_2}$, there is $u_2\in H^1(\Omega)$
satisfying $L_{A_2,q_2}u_2=0$ in $\Omega$ such that
\[
u_1'=u_2\quad \textrm{and}\quad \p_\nu u_1'+i(A_1\cdot \nu) u_1'=\p_\nu u_2+i(A_2\cdot \nu)u_2\quad \textrm{on}\quad \p \Omega. 
\]
Here $\nu$ is the unit outer normal to $\p\Omega$.  Setting 
\[
u_2'=\begin{cases} u_2\quad \textrm{in}\quad \Omega,\\
u_1'\quad \textrm{in}\quad B\setminus\Omega,
\end{cases}
\]
and using the fact that $\p \Omega$ is $C^1$, we see that $u_2'\in H^1(B)$. 
Furthermore, for $\psi\in C^\infty_0(B)$,  we get 
\begin{align*}
\langle L_{A_2,q_2}u_2',&\psi \rangle_{B}=  \int_{\Omega}(\nabla u_2\cdot\nabla \psi+ A_2\cdot (Du_2)\psi-A_2u_2\cdot D\psi+ (A_2^2+q_2)u_2\psi)dx\\
&+\int_{B\setminus\Omega}(\nabla u_1'\cdot\nabla \psi+ A_1\cdot (Du_1')\psi-A_1u_1'\cdot D\psi+ (A_1^2+q_1)u_1'\psi)dx\\
&=\langle \p_\nu u_2+i(A_2\cdot \nu)u_2 ,\psi \rangle_{\p \Omega}- \langle \p_\nu u_1'+i(A_1\cdot \nu) u_1',\psi \rangle_{\p \Omega}=0,
\end{align*}
which shows that $C'_{A_1,q_1}\subset C'_{A_2,q_2}$. The same argument in the other direction gives the claim. 
\end{proof}

Using Proposition \ref{prop_Cauchy_data}, the fact that $\psi|_{\p B}=0$, and \eqref{eq_invariance_cauchy}, we obtain that 
\[
C'_{A_1,q_1}=C'_{A_2,q_2}=C'_{A_2+\nabla\psi,q_2}=C'_{A_1,q_2}.
\]
Thus, by Proposition \ref{prop_eq_int_identity} we obtain the integral identity \eqref{eq_int_identity}, which in our case takes the following form, 
\begin{equation}
\label{eq_int_identity_new_2}
\int_B(q_1-q_2)u_1\overline{u_2}dx=0,
\end{equation}
for any $u_1,u_2\in H^1(B)$ satisfying $L_{A_1,q_1}u_1=0$ in $B$ and $L_{\overline{A_1},\overline{q_2}}u_2=0$ in $B$, respectively.  

Let us choose $u_1$ and $u_2$ to be the complex geometric optics solutions in $B$, given by \eqref{eq_u_1} and \eqref{eq_u_2}, respectively.  Notice that in our case, in the definition \eqref{eq_u_2} of $u_2$, the  function $\Phi_2^\sharp(x,-\mu_1+i\mu_2;h)$ converges uniformly in $\R^n$ to $N_{-\mu_1+i\mu_2}^{-1}(-i(-\mu_1+i\mu_2)\cdot \overline{A_1})$ as $h\to 0$.  Hence, it follows from \eqref{eq_sum_phase} that 
$\Phi_1^\sharp (x,\mu_1+i\mu_2;h)+ \overline{\Phi_2^\sharp (x,-\mu_1+i\mu_2;h)}$ converges uniformly in $\R^n$ to $0$ as $h\to 0$. 

Plugging $u_1$ and $u_2$ in \eqref{eq_int_identity_new_2} gives
\[
\int_B(q_1-q_2)e^{ix\cdot\xi} e^{\Phi_1^\sharp+\overline{\Phi_2^\sharp}}dx=-\int_B(q_1-q_2)e^{ix\cdot\xi} (e^{\Phi_1^\sharp}\overline{r_2}+r_1e^{\overline{\Phi_2^\sharp}}+r_1\overline{r_2})dx.
\]
Letting $h\to 0$, and using \eqref{eq_est_phi_1}, \eqref{eq_est_r_1}, \eqref{eq_est_phi_2}, and  \eqref{eq_est_r_2}, we get 
\[
\int_B(q_1-q_2)e^{ix\cdot\xi}dx=0,
\] 
and therefore, $q_1=q_2$ in $\Omega$. The proof of Theorem \ref{thm_main} is complete.

\section{Boundary determination of the magnetic potential}

\label{sec_boundary_rec}

When recovering the magnetic potentials in Theorem \ref{thm_main}, an important step consists in determining the boundary values of the tangential components of the magnetic potentials. The purpose of this section is to carry out this step by adapting the method of  \cite{Brown_Salo_2006}.  Compared with the latter work, here we do not assume that the Dirichlet problem for the magnetic Schr\"odinger operator in $\Omega$ is well posed. 

To circumvent the difficulty related to the fact that zero may be a Dirichlet eigenvalue,  we shall use the solvability result for the magnetic Schr\"odinger operator given in Proposition \ref{prop_solvability},  which is based on a Carleman estimate. We have learned of the idea of using a Carleman estimate to handle the case when zero is a Dirichlet eigenvalue from the work \cite{Salo_Tzou_2009} on the Dirac operator. 

The following proposition is an extension of the result of \cite{Brown_Salo_2006} in the sense that the well-posedness of the Dirichlet problem for the magnetic Schr\"odinger operator in $\Omega$ is no longer assumed.

\begin{prop} 
\label{prop_boundary_rec}
Let $\Omega\subset \R^n$, $n\ge 3$, be a bounded domain with $C^1$ boundary, and 
let $A_j\in C(\overline{\Omega},\C^n)$ and $q_j\in L^\infty(\Omega,\C)$, $j=1,2$.  Assume that $C_{A_1,q_1}=C_{A_2,q_2}$. 
Then 
$\tau\cdot (A_1- A_2)(x_0)=0$,
for all points $x_0\in \p \Omega$ and all unit tangent vectors $\tau\in T_{x_0}(\p \Omega)$. 
\end{prop}

\begin{proof}
We shall follow  \cite{Brown_Salo_2006} closely.  First an application of Proposition \ref{prop_eq_int_identity} allows us to conclude that following integral identity holds,
\begin{equation}
\label{eq_int_identity_4_1}
\int_\Omega i(A_1-A_2)\cdot (u_1\nabla \overline{u_2}-\overline{u_2}\nabla u_1)dx+\int_\Omega(A_1^2-A_2^2+q_1-q_2)u_1\overline{u_2}dx=0.
\end{equation}
Here $u_1,u_2\in H^1(\Omega)$ satisfy $L_{A_1,q_1}u_1=0$ in $\Omega$ and $L_{\overline{A_2},\overline{q_2}}u_2=0$ in $\Omega$, respectively.   

The idea now  is to construct some special solutions to the magnetic Schr\"odinger equations, whose boundary values have an oscillatory behavior while becoming increasingly concentrated near a given point on the boundary of 
$\Omega$.  Substituting these solutions in \eqref{eq_int_identity_4_1} will allow us to recover the tangential components of the magnetic potentials along the boundary.

As $\Omega$ is a $C^1$-domain, it has a defining function $\rho\in C^1(\R^n,\R)$ such that $\Omega=\{x\in \R^n:\rho(x)>0\}$, $\p \Omega=\{x\in \R^n:\rho(x)=0\}$, and 
$\nabla \rho$ does not vanish on $\p \Omega$. We fix $x_0\in \p \Omega$, and a unit vector $\tau$, which is tangent to $\p \Omega$ at $x_0$. We normalize $\rho$ so that $\nabla \rho(x_0)=-\nu(x_0)$ where $\nu$ is the unit outer normal to $\p \Omega$. 
By an affine change of coordinates we may assume that $x_0$ is the origin and $\nu(x_0)=-e_n$, and therefore, $\nabla \rho(0)=e_n$. 

Let $\omega(t)$, $t\ge 0$, be a modulus of continuity for $\nabla \rho$, which is  an increasing continuous function, such that  
$\omega(0)= 0$. 
Let $\eta\in C^\infty_0(\R^n,\R)$ be a function such that $\supp (\eta)\subset B(0,1/2)$, and
\[
\int_{\R^{n-1}}\eta(x',0)^2dx'=1,
\]
where $B(0,1/2)$ is a ball of radius $1/2$, centered at $0$, and $x'=(x_1,\dots,x_{n-1})$.
We set $\eta_M(x)=\eta(Mx',M\rho(x))$, for $M>0$. Hence, for $M>0$ large enough, $\supp(\eta_M)\subset B(0,1/M)$. 
Following \cite{Brown_Salo_2006}, for $N>0$, we define $v_0$ by
\begin{equation}
\label{eq_7_1}
v_0(x)=\eta_M(x)e^{N(i\tau\cdot x-\rho(x))}.
\end{equation}
The function $v_0$ is of class $C^1$ with $\supp(v_0)\subset B(0,1/M)$.  Following   \cite{Brown_Salo_2006}, we relate the large parameters $N$ and $M$  by the equation 
\begin{equation}
\label{eq_7_2}
M^{-1}\omega(M^{-1})=N^{-1}. 
\end{equation}
 Since $\omega(t)\to 0$ as $t\to +0$, there is $M_0$ such that $\omega(M^{-1})<1$ for $M>M_0$. We shall assume that $M>M_0$ and therefore, $N>M$.

Let $v_1\in H^1(\Omega)$ be the solution to the following Dirichlet problem for the Laplacian,
\begin{align*}
-\Delta v_1&=\Delta v_0, \quad\textrm{in}\quad \Omega,\\
v_1|_{\p \Omega}&=0.   
\end{align*} 

In what follows we shall need Hardy's inequality,
\begin{equation}
\label{eq_Hardy}
\int_{\Omega}|f(x)/\delta(x)|^2 dx\le C\int_{\Omega}|\nabla f(x)|^2dx,
\end{equation}
where $f\in H^1_0(\Omega)$ and $\delta(x)$ denotes the distance from $x$ to the boundary of $\Omega$.  The constant $C$ in \eqref{eq_Hardy} depends only on  the geometry of $\Omega$.

We shall also need the following estimates, obtained in  \cite{Brown_Salo_2006},
\begin{equation}
\label{eq_7_3}
\|v_0\|_{L^2(\Omega)}\le CM^{(1-n)/2}N^{-1/2},
\end{equation}
\begin{equation}
\label{eq_7_4}
\|v_1\|_{L^2(\Omega)}\le CM^{(1-n)/2}N^{-1/2},
\end{equation}
\begin{equation}
\label{eq_7_15}
\|\nabla v_1\|_{L^2(\Omega)}\le C\omega(M^{-1})N^{1/2}M^{(1-n)/2},
\end{equation}
\begin{equation}
\label{eq_Brown_Salo_2_18}
\|\delta\nabla v_0\|_{L^2(\Omega)}\le C M^{(1-n)/2}N^{-1/2},
\end{equation}
and 
\begin{equation}
\label{eq_Brown_Salo_2_20}
\|\delta\nabla (v_0+v_1)\|_{L^2(\Omega)}\le C M^{(1-n)/2}N^{-1/2}. 
\end{equation}

Next we would like to show the  existence of a solution $u_1\in H^1(\Omega)$ to the magnetic Schr\"odinger equation
\begin{equation}
\label{eq_7_10}
L_{A_1,q_1}u_1=0\quad \textrm{in}\quad \Omega,
\end{equation}
of the form 
\begin{equation}
\label{eq_7_11}
u_1=v_0+v_1+r_1,
\end{equation}
with 
\begin{equation}
\label{eq_7_12}
\|r_1\|_{H^1(\Omega)}\le  CM^{(1-n)/2}N^{-1/2}.
\end{equation}
To that end, plugging \eqref{eq_7_11} into \eqref{eq_7_10}, we obtain the following equation for $r_1$,
\[
L_{A_1,q_1} r_1=-A_1\cdot D(v_0+v_1)- m_{A_1}(v_0+v_1)  -(A_1^2+q_1)(v_0+v_1)\quad \textrm{in}\quad \Omega.
\]
Applying Proposition \ref{prop_solvability} with $h>0$ small but fixed, we conclude the existence of $r_1\in H^1(\Omega)$ such that
\begin{equation}
\label{eq_right_hand_r_1}
\|r_1\|_{H^1(\Omega)}\le C\|A_1\cdot D(v_0+v_1)+ m_{A_1}(v_0+v_1)  +(A_1^2+q_1)(v_0+v_1)\|_{H^{-1}(\Omega)}.
\end{equation}
Let us now compute the norm in the right hand side of \eqref{eq_right_hand_r_1}. To that end, let $\psi\in C^\infty_0(\Omega)$. Then using \eqref{eq_Brown_Salo_2_20} and \eqref{eq_Hardy}, we get
\begin{equation}
\label{eq_r_1_est_1}
\begin{aligned}
|\langle A_1\cdot D(v_0+v_1) ,\psi\rangle_{\Omega}|
\le \|A_1\|_{L^\infty(\Omega)}\|\delta\nabla(v_0+v_1) \|_{L^2(\Omega)}\|\psi/\delta\|_{L^2(\Omega)}\\
\le C M^{(1-n)/2}N^{-1/2}\|\nabla \psi\|_{L^2(\Omega)}\le C M^{(1-n)/2}N^{-1/2}\| \psi\|_{H^1(\Omega)}.
\end{aligned}
\end{equation}
Using \eqref{eq_7_3} and \eqref{eq_7_4}, we obtain that 
\begin{equation}
\label{eq_r_1_est_2}
\begin{aligned}
|\langle & m_{A_1}(v_0+v_1)  ,\psi\rangle_{\Omega}|=\bigg| \int_{\Omega} A(v_0+v_1)\cdot D\psi dx\bigg|\\
&\le \|A_1\|_{L^\infty(\Omega)}\|v_0+v_1\|_{L^2(\Omega)}\|\nabla \psi\|_{L^2(\Omega)}\le C M^{(1-n)/2}N^{-1/2}\| \psi\|_{H^1(\Omega)},
\end{aligned}
\end{equation}
and 
\begin{equation}
\label{eq_r_1_est_3}
\begin{aligned}
|\langle  (A_1^2+q_1)(v_0+v_1) ,\psi\rangle_{\Omega}|\le C M^{(1-n)/2}N^{-1/2}\| \psi\|_{H^1(\Omega)}.
\end{aligned}
\end{equation}
The estimate  \eqref{eq_7_12} follows from \eqref{eq_right_hand_r_1}, \eqref{eq_r_1_est_1},  \eqref{eq_r_1_est_2}, and \eqref{eq_r_1_est_3}.

Similarly, there exists a solution $u_2\in H^1(\Omega)$ of $L_{\overline{A_2},\overline{q_2}}u_2=0$ in $\Omega$ of the form 
\begin{equation}
\label{eq_7_16}
u_2=v_0+v_1+r_2,
\end{equation}
where $r_2\in H^1(\Omega)$ satisfies \eqref{eq_7_12}.  

The next step is to substitute $u_1$ and $u_2$, given by \eqref{eq_7_11} and \eqref{eq_7_16} into the identity \eqref{eq_int_identity_4_1}, multiply it by $M^{n-1}$ and compute the limit as $M\to \infty$. 
We write 
\begin{equation}
\label{eq_def_I}
I:= M^{n-1}\int_\Omega i (A_1-A_2)\cdot (u_1\nabla\overline{u_2}-\overline{u_2}\nabla u_1)dx=I_1+I_2+I_3+I_4+I_5,
\end{equation}
where 
\begin{align*}
I_1&= M^{n-1}\int_\Omega i (A_1-A_2)\cdot (v_0\nabla\overline{v_0}-\overline{v_0}\nabla v_0)dx,\\
I_2&=M^{n-1}\int_\Omega i (A_1-A_2)\cdot (u_1\nabla\overline{v_1}-\overline{u_2}\nabla v_1)dx,\\
I_3&= M^{n-1}\int_\Omega i (A_1-A_2)\cdot (u_1\nabla\overline{r_2}-\overline{u_2}\nabla r_1)dx,\\
I_4&= M^{n-1}\int_\Omega i (A_1-A_2)\cdot (v_1\nabla\overline{v_0}-\overline{v_1}\nabla v_0)dx,\\
I_5&= M^{n-1}\int_\Omega i (A_1-A_2)\cdot (r_1\nabla\overline{v_0}-\overline{r_2}\nabla v_0)dx.
\end{align*}
Let us compute $\lim_{M\to \infty}I_1$. To that end we have
\begin{equation}
\label{eq_v_0_grad}
\nabla v_0=(\nabla\eta_M(x)+\eta_M(x)N(i\tau-\nabla\rho(x))e^{N(i\tau\cdot x-\rho(x))},
\end{equation}
and 
\[
v_0\nabla\overline{v_0}-\overline{v_0}\nabla v_0=-2i\eta_M^2(x)e^{-2N\rho(x)}N\tau.
\]
Thus, using the following formulas 
\[
\lim_{M\to \infty} M^{n-1}N\int_\Omega e^{-2N\rho(x)}\eta_M^2(x)dx=\frac{1}{2}\int_{\R^{n-1}}\eta^2(x',0)dx'=\frac{1}{2},
\]
\begin{equation}
\label{eq_Brown_Salo_l21}
\int_\Omega e^{-2N\rho(x)}\eta_M^2(x)dx\le CM^{1-n}N^{-1},
\end{equation}
obtained in \cite{Brown_Salo_2006}, and the fact that $A_1,A_2\in C(\overline{\Omega})$, we get
\begin{equation}
\label{eq_7_18}
\begin{aligned}
\lim_{M\to \infty} I_1&=2  ((A_1-A_2)(0)\cdot\tau)\lim_{M\to \infty}M^{n-1}N\int_\Omega \eta^2_M(x)e^{-2N\rho(x)}dx\\
&+2\lim_{M\to \infty}M^{n-1}N\int_\Omega ((A_1-A_2)(x)-(A_1-A_2)(0))\cdot\tau \eta^2_M(x)e^{-2N\rho(x)}dx\\
&=(A_1-A_2)(0)\cdot\tau.
\end{aligned}
\end{equation}

Now it follows from  \eqref{eq_7_3}, \eqref{eq_7_4} and \eqref{eq_7_12}  that
\begin{equation}
\label{eq_7_17}
\|u_j\|_{L^2(\Omega)}\le CM^{(1-n)/2}N^{-1/2},\quad j=1,2.
\end{equation}

From the estimates \eqref{eq_7_15} and \eqref{eq_7_17}, we obtain that 
\begin{equation}
\label{eq_I_2}
|I_2|\le M^{n-1}\|A_1-A_2\|_{L^\infty(\Omega)}(\|u_1\|_{L^2(\Omega)}+ \|u_2\|_{L^2(\Omega)})\|\nabla v_1\|_{L^2(\Omega)}
\le C\omega(M^{-1}).
\end{equation}

By \eqref{eq_7_12} and \eqref{eq_7_17}, we have
\begin{equation}
\label{eq_I_3}
\begin{aligned}
|I_3|&\le M^{n-1}\|A_1-A_2\|_{L^\infty(\Omega)}(\|u_1\|_{L^2(\Omega)}\|\nabla r_2\|_{L^2(\Omega)}
+ \|u_2\|_{L^2(\Omega)}\|\nabla r_1\|_{L^2(\Omega)})\\
&\le CN^{-1}=CM^{-1}\omega(M^{-1}).
\end{aligned}
\end{equation}

Using the fact that $v_1\in H^1_0(\Omega)$, Hardy's inequality \eqref{eq_Hardy}, and the estimates  \eqref{eq_Brown_Salo_2_18},  and \eqref{eq_7_15}, we get 
\begin{equation}
\label{eq_I_4}
\begin{aligned}
|I_4|&\le 2M^{n-1}\|A_1-A_2\|_{L^\infty(\Omega)}\|v_1/\delta\|_{L^2(\Omega)}\|\delta\nabla v_0\|_{L^2(\Omega)}\\
&\le CM^{n-1}\|\nabla v_1\|_{L^2(\Omega)}\|\delta\nabla v_0\|_{L^2(\Omega)}\le C\omega(M^{-1}).
\end{aligned}
\end{equation}

Let us now estimate $|I_5|$. To that end we shall first obtain an estimate for $\|\nabla v_0\|_{L^2(\Omega)}$. 
We have, using \eqref{eq_v_0_grad}, 
\begin{equation}
\label{eq_nabla_v_0_1}
\begin{aligned}
|\nabla v_0(x)|^2=&e^{-2N\rho(x)}(|\nabla\eta_M(x)|^2-2(\nabla \eta_M(x))\cdot  \eta_M(x) N\nabla \rho(x)\\
&+\eta_M^2(x)N^2(1+|\nabla\rho(x)|^2))\\
&\le 2 e^{-2N\rho(x)}|\nabla\eta_M(x)|^2+ 2 e^{-2N\rho(x)} \eta_M^2(x)N^2(1+|\nabla\rho(x)|^2)).
\end{aligned}
\end{equation}
By \eqref{eq_Brown_Salo_l21} we get 
\begin{equation}
\label{eq_nabla_v_0_2}
\int_\Omega e^{-2N\rho(x)} \eta_M^2(x)N^2(1+|\nabla\rho(x)|^2))dx\le C M^{1-n} N. 
\end{equation}
Next we have
\begin{align*}
|\nabla\eta_M(x)|^2=&M^2\sum_{j=1}^{n-1}\bigg( (\p_{y_j} \eta)(Mx',M\rho(x)) + (\p_{x_j}\rho)(x)(\p_{y_n}\eta) (Mx',M\rho(x))  \bigg)^2\\
&+ M^2 (\p_{x_n}\rho)^2(x)(\p_{y_n}\eta)^2 (Mx',M\rho(x)),
\end{align*}
and therefore, using the estimates
\[
\int_\Omega e^{-2N\rho(x)}(\p_{y_j}\eta)^2 (Mx',M\rho(x))dx\le CM^{1-n}N^{-1},
\]
obtained in \cite[Lemma 2.1]{Brown_Salo_2006}, we get 
\begin{equation}
\label{eq_nabla_v_0_3}
\int_\Omega e^{-2N\rho(x)} |\nabla\eta_M(x)|^2 dx\le CM^2 M^{1-n}N^{-1}\le C M^{1-n} N.
\end{equation}
Thus, it follows from \eqref{eq_nabla_v_0_1}, \eqref{eq_nabla_v_0_2}, and \eqref{eq_nabla_v_0_3} that 
 \begin{equation}
\label{eq_nabla_v_0}
\|\nabla v_0\|_{L^2(\Omega)}\le CM^{(1-n)/2} N^{1/2}.
\end{equation}
A direct computation also  shows that 
\begin{equation}
\label{eq_v_0_boundary}
\| v_0\|_{L^2(\p \Omega)}\le CM^{(1-n)/2}.
\end{equation}

Let us extend $A:=A_1-A_2\in C(\overline{\Omega},\C^n)$ to a continuous compactly supported vector field on the whole of $\R^n$. We consider the regularization $A^\sharp=A*\Psi_\tau\in C_0^\infty(\R^n,\C^n)$. Here $\tau>0$ is small and 
$\Psi_\tau(x)=\tau^{-n}\Psi(x/\tau)$ is the usual mollifier with $\Psi\in C^\infty_0(\R^n)$, $0\le \Psi\le 1$, and 
$\int \Psi dx=1$.   Then the following estimates hold, 
\begin{equation}
\label{eq_flat_est_new}
\|A-A^\sharp\|_{L^\infty(\R^n)}=o(1), \quad \tau \to 0,
\end{equation}
\begin{equation}
\label{eq_flat_est_2_new}
 \|\p^\alpha A^\sharp\|_{L^\infty(\R^n)}=\mathcal{O}(\tau^{-|\alpha|}),\quad \tau\to 0, \quad \textrm{for all}\quad \alpha.
\end{equation}
 Let us consider 
  \begin{align*}
J:= M^{n-1}\int_{\Omega} A\cdot r_1\nabla \overline{v_0} dx=J_1+J_2,
 \end{align*}
where
\begin{align*}
J_1= M^{n-1}\int_{\Omega} (A-A^\sharp)\cdot r_1\nabla \overline{v_0} dx,\quad 
J_2= M^{n-1}\int_{\Omega} A^\sharp\cdot r_1\nabla \overline{v_0} dx.
\end{align*}
Using \eqref{eq_7_12},  \eqref{eq_nabla_v_0}, and \eqref{eq_flat_est_new}, we get 
\[
|J_1|\le M^{n-1}\|A-A^\sharp\|_{L^\infty(\Omega)}\|r_1\|_{L^2(\Omega)}\|\nabla v_0\|_{L^2(\Omega)}\le C\|A-A^\sharp\|_{L^\infty(\Omega)}=o(1),
\]
as $ \tau\to 0$. To estimate $J_2$, it is no longer sufficient to use the bound \eqref{eq_nabla_v_0}, and therefore, we shall integrate by parts. We get 
\[
J_2=J_{2,1}+J_{2,2}+J_{2,3},
\]
where  
\begin{align*}
 &J_{2,1}= -M^{n-1}\int_{\Omega} (\nabla \cdot A^\sharp) r_1 \overline{v_0} dx,\quad J_{2,2} =- M^{n-1}\int_{\Omega} A^\sharp \cdot (\nabla r_1)\overline{v_0} dx,\\
 &J_{2,3}= M^{n-1}\int_{\p \Omega} (A^\sharp\cdot\nu) r_1\overline{v_0} dS.
\end{align*}

By \eqref{eq_7_3}, \eqref{eq_7_12}, and \eqref{eq_flat_est_2_new}, we have
\[
|J_{2,1}|\le M^{n-1}\|\nabla \cdot A^\sharp\|_{L^\infty(\Omega)}\|r_1\|_{L^2(\Omega)}\|v_0\|_{L^2(\Omega)}\le C\tau^{-1} N^{-1},
\] 
and 
\[
|J_{2,2}|\le M^{n-1}\|A^\sharp\|_{L^\infty(\Omega)} \|\nabla r_1\|_{L^2(\Omega)}\|v_0\|_{L^2(\Omega)} \le CN^{-1}. 
\]
Using the trace theorem, the estimates \eqref{eq_7_12} and \eqref{eq_v_0_boundary}, we get
\[
|J_{2,3}|\le CM^{n-1}\|A^\sharp\cdot\nu\|_{L^\infty(\p \Omega)}\|r_1\|_{H^1(\Omega)}\|v_0\|_{L^2(\p \Omega)}\le CN^{-1/2}. 
\]
Choosing $\tau=N^{-1/2}$, we conclude that $|J|=o(1)$ as $M\to \infty$, and therefore, 
\begin{equation}
\label{eq_I_5}
|I_5|=o(1),\quad M\to \infty. 
\end{equation}

Hence, for $I$, defined by \eqref{eq_def_I}, using \eqref{eq_7_18}, \eqref{eq_I_2}, \eqref{eq_I_3}, \eqref{eq_I_4}, and  \eqref{eq_I_5}, we get 
\begin{equation}
\label{eq_4_1_first}
\lim_{M\to \infty} I=(A_1-A_2)(0)\cdot \tau. 
\end{equation}

Furthermore, for $u_1$ and $u_2$, given by \eqref{eq_7_11} and \eqref{eq_7_16}, respectively, using \eqref{eq_7_17},  we obtain that 
\begin{equation}
\label{eq_4_1_second}
M^{n-1}\bigg|\int_\Omega (A_1^2-A_2^2+q_1-q_2)u_1\overline{u_2}dx \bigg|\le CM^{n-1}\|u_1\|_{L^2(\Omega)}\|u_2\|_{L^2(\Omega)}\le CN^{-1}.
\end{equation}

Thus, we conclude from \eqref{eq_int_identity_4_1} with the help of \eqref{eq_4_1_first} and \eqref{eq_4_1_second} that $(A_1-A_2)(0)\cdot \tau=0$. The proof is complete. 
\end{proof}

\section*{Acknowledgements}  

The research of K.K. is partially supported by the
Academy of Finland (project 255580).   The research of
G.U. is partially supported by the National Science Foundation.

\end{document}